\documentclass[12pt]{amsart}
\usepackage{amsfonts}
\usepackage{mathrsfs}
\usepackage{amsmath}
\usepackage{amsmath, amsthm, amssymb}
\usepackage{color}
\numberwithin{equation}{section}
\newcommand{\reff}[1]{(\ref{#1})}

\newcommand{\C}{\mathbb C}
\newcommand{\qq}[2]{\ensuremath{(#1;#2)_\infty}}
\newcommand{\res}{\mathrm{res}}

\newcommand{\T}{\theta}

\renewcommand{\(}{\left( }
\renewcommand{\)}{\right) }

\renewcommand{\theequation}{\theequation. \arabic{equation}}
\numberwithin{equation}{section}
\newtheorem{thm}{Theorem}[section]

\newtheorem{prop}[thm]{Proposition}
\newtheorem{defn}[thm]{Definition}

\def\squarebox#1{\hbox to #1{\hfill\vbox to #1{\vfill}}}

\begin{document}
\title[On the Kronecker theta function]
{Kronecker theta function and a decomposition theorem for theta functions~I}

\author{Zhi-Guo Liu}
\date{\today}
\address{School of Mathematical Sciences and  Shanghai  Key Laboratory of PMMP, East China Normal University, 500 Dongchuan Road,
	Shanghai 200241, P. R. China} \email{zgliu@math.ecnu.edu.cn;
	liuzg@hotmail.com}

\thanks{This work was supported by the National Science Foundation of China (Grant Nos. 11971173 and 11571114) and Science and Technology Commission of Shanghai Municipality (Grant No. 13dz2260400).}

\begin{abstract}
The Kronecker theta function is a quotient of the Jacobi theta functions, which is also a special case of Ramanujan's $_1\psi_1$ summation.  Using the Kronecker theta function as building blocks, we prove  a decomposition theorem for theta functions.  This decomposition theorem is the common source of a large number of  theta function identities. Many striking theta function identities, both classical and new,  are derived from this decomposition theorem with ease. A new addition formula for theta functions is established.  Several known results in the theory of elliptic theta functions due to Ramanujan, Weierstrass, Kiepert, Winquist and Shen among others are revisited.  A  curious trigonometric identities is proved.
\end{abstract}
\dedicatory{ In memory of Professor  Richard Askey }
\maketitle
\section{Introduction and preliminary}
\noindent
 Let $q$ be a complex number with $|q|<1$. For any integer $n$,  the $q$-shifted factorial $(a; q)_n$
  is defined by
 \begin{equation}\label{kron:eqn1}
 (a; q)_n
 =\frac{(a; q)_\infty}{(aq^n; q)_\infty},~ \text{where}\quad 
  (a; q)_\infty=\prod_{k=0}^\infty (1-aq^k).
 \end{equation}
 If $n$ is an integer and $m$ a positive integer, sometimes we also adopt the following compact notation for multiple $q$-shifted factorial:
 \begin{equation}\label{kron:eqn2}
 	(a_1, a_2,...,a_m;q)_n=(a_1;q)_n(a_2;q)_n ... (a_m;q)_n.
 \end{equation}
Unless  otherwise  stated,  it is assumed throughout  that $q$=$\exp (2\pi i\tau)$, where
$\text{Im}\tau>0.$ 
\begin{defn}\label{thetadefn}
The Jacobi theta functions $\theta_1, \theta_2, \theta_3$ and $\theta_4$ are defined as
\begin{align*}
\theta_1(z| \tau)&=2\sum_{n=0}^\infty (-1)^n q^{(2n+1)^2/8} 
\sin (2n+1)z,\\
\T_2(z|\tau)&=2\sum_{n=0}^\infty q^{(2n+1)^2/8} \cos (2n+1)z, \\
\T_3(z|\tau)&=1+2\sum_{n=1}^\infty q^{{n^2}/{2}} \cos 2nz, \\
\T_4(z|\tau)&=1+2\sum_{n=1}^\infty (-1)^n q^{n^2/2} \cos 2nz.
\end{align*}
\end{defn}
\noindent These series converges for all complex values of $z$ whenever $\tau$
has positive imaginary part (equivalently, when $|q|<1$). Furthermore, these series converge absolutely and uniformly on
compact subsets and so are entire functions of $z$.

The infinite representations of the Jacobi theta functions can be derived by using the Jacobi triple product identity, which can be found in any  standard textbook in elliptic theta functions.
\begin{prop}\label{infiniteprod}  Let $\T_1, \T_2, \T_3$ and $\T_4$ be defined as 
	in Definition~\ref{thetadefn}. Then we have 
\begin{align*}
\T_1(z|\tau)&=2q^{1/8}(\sin z )\prod_{n=1}^\infty (1-q^n)(1-q^ne^{2iz}) (1-q^ne^{-2iz}),\\
\T_2(z|\tau)&=2q^{1/8}(\cos z )\prod_{n=1}^\infty (1-q^n)(1+q^ne^{2iz}) (1+q^ne^{-2iz}),\\
\T_3(z|\tau)&=\prod_{n=1}^\infty (1-q^n)(1+q^{n-1/2} e^{2iz}) (1+q^{n-1/2}e^{-2iz}),\\
\T_4(z|\tau)&=\prod_{n=1}^\infty (1-q^n)(1-q^{n-1/2} e^{2iz}) (1-q^{n-1/2}e^{-2iz}).	
\end{align*}
\end{prop}
Theta functions $\T_1, \T_2, \T_3$ and $\T_4$  are not  elliptic functions. The following functional equations will be often used in this paper.
\begin{prop}\label{doubleperiods}
With respect to the
(quasi) periods $\pi$ and $\pi\tau$, the Jacobi theta function $\T_1(z|\tau)$
satisfies the functional equations
\begin{align*}
\theta_1(z | \tau)&=-\theta_1(z+\pi | \tau)=-q^{1/2} e^{2iz}
\theta_1(z+\pi\tau | \tau),\\
\theta_2(z | \tau)&=-\theta_2(z+\pi | \tau)=q^{1/2} e^{2iz}
\theta_2(z+\pi\tau | \tau), \\
\theta_3(z | \tau)&=\theta_3(z+\pi | \tau)=q^{1/2} e^{2iz}
\theta_3(z+\pi\tau | \tau),\\
\theta_4(z | \tau)&=\theta_4(z+\pi | \tau)=-q^{1/2} e^{2iz}
\theta_4(z+\pi\tau | \tau).
\end{align*}
\end{prop}
The four Jacobi theta functions are mutually related, and starting from one of them we may obtain the other three by simple calculation. For example, we have
\begin{prop}\label{halfperiods} Theta functions $\T_1, \T_2, \T_3$ and $\T_4$ satisfy the relations
\begin{align*}
\T_1(z+\pi/2|\tau)&=\T_2(z|\tau),\\
\T_2(z+\pi \tau/2|\tau)&=iq^{-1/8} e^{-iz} \T_4(z|\tau), \\
\T_3(z+(\pi+\pi\tau)/2|\tau)&=q^{-1/8} e^{-iz} \T_3(z|\tau).
\end{align*}
\end{prop}
It is easily seen that in the fundamental periodic parallelogram
given by
\begin{equation}\label{kron:eqn3}
\prod=\{x\pi+y\pi\tau| 0\le x< 1, 0\le y< 1\}
\end{equation}
the zero points of $\theta_1(z|\tau)$ is  at
$z=0$.

The Jacobi theta function $\theta_1(z|\tau)$ is an entire, quasi-doubly periodic function of $z$, 
and  regarded as the two dimensional version of the  sine  function.  The set of zero points of $\theta_1(z|\tau)$  form a lattice $\Lambda$, which is given by 
\begin{equation}\label{kron:eqn4}
\Lambda=\{ m \pi+ n\pi \tau: (m, n) \in \mathbb{Z}^2 \}.
\end{equation}

\begin{defn} \label{qequiv}
If the difference of two complex numbers is in the set $\Lambda$,  then these two complex numbers are 
said to be equivalent modulo  $ \Lambda$, and if the difference of two complex numbers is not in the set $\Lambda$,  then these two complex numbers are 
said to be inequivalent modulo  $ \Lambda$.
\end{defn}	
We will use the prime to denote partial differential operator with respect
to the first variable of theta functions.  From the infinite product representation of $\theta_1,$ we easily find that
\begin{equation}\label{kron:eqn5}
\theta_1'(0|\tau)=2q^{1/8} \prod_{n=1}^{\infty} (1-q^n)^3=2\eta^3(\tau),
\end{equation}
where $\eta(\tau)$ is the well-known Dedekind eta-function which is defined as
\begin{equation}\label{kron:eqn6}
\eta(\tau)=q^{1/24} \prod_{n=1}^\infty (1-q^n).
\end{equation}
Ramanujan's $_1\psi_1$ summation formula \cite[p.502]{AAR1999} is an extension of Jacobi's triple product identity, which states that for
$|q|<1$ and $|b/a|<|z|<1$,
\begin{align}\label{kron:eqn7}
\sum_{k=-\infty}^\infty \frac{(a; q)_k}{(b; q)_k}z^k =\frac{(q, b/a, az, q/az; q)_\infty}{(b, q/a, z, b/az; q)_\infty}.
\end{align}
There are several different proofs of this summation formula in the literature. See for example, Andrews and Askey \cite{AndrewsAskey1978}, Berndt \cite[pp. 15--17]{Berndt} and Ismail \cite{Ismail1977}.

Taking $b=aq$ in the Ramanujan $_1\psi_1$ summation and noting that 
$
(a; q)_k/(aq; q)_k=(1-a)/(1-aq^k)
$
we deduce that for $a$ being neither zero nor an integral power of $q$ and $|q|<|z|<1$, 
\begin{align}\label{kron:eqn8}
\sum_{k=-\infty}^\infty \frac{z^k}{1-aq^k}=\frac{\qq{q}{q}^2 (az, q/az; q)_\infty}{(a, q/a, z, q/z; q)_\infty}.
\end{align}

Replacing $z$ by $e^{2iy}$ and $a$ by $e^{2iz}$  in (\ref{kron:eqn8}) and noting that the infinite product representation of $\theta_1(z|\tau)$ in Proposition~\ref{infiniteprod}, we arrive at the following remarkable theta function identity, which was known  to Kronecker \cite{Kronecker1881}, \cite[pp. 309--318]{KroneckerC1929}:
\begin{align}\label{kron:eqn9}
\sum_{k=-\infty}^\infty \frac{e^{2kiy}}{1-q^k e^{2iz}}=\frac{i\T_1'(0|\tau)\T_1(y+z|\tau)}{2\T_1(y|\tau)\T_1(z|\tau)},
\end{align}
where  $q=e^{2\pi i\tau}$ with $\text{Im} \tau>0.$ 

 It should be pointed out that Shen \cite{Shen1994} used this identity to derive the Fourier series of all  the twelve Jacobian elliptic functions and also used these  Fourier series  to set up 24 Lambert series representations for quotients of theta functions.

Now we give the definition of the  Kronecker theta  function. 
\begin{defn}\label{kroneckerthetadefn}
	If $y$ and $z$ are not zero points of $\T_1(z|\tau),$ we define the Kronecker function $K_y(z|\tau)$ as
	\begin{equation}\label{kron:eqn10}
	K_y(z|\tau)=\frac{\T'_1(0|\tau)\T_1(z+y|\tau)}{\T_1(z|\tau)\T_1(y|\tau)}.
	\end{equation}
\end{defn}
In this paper we will use $\res(f; \alpha)$ to denote the residue of $f(z)$ at $z=\alpha$.  The principal result of this paper  is the following decomposition  theorem for theta  functions. This paper is motivated by Kronecker \cite{Kronecker1881}, Lewis and Liu \cite{LewisLiu2001} and Shen \cite{Shen1994}.
\begin{thm}\label{liumainthm:a}
 Suppose that $f(z)$ is an meromorphic function of $z$ which has only  simple poles and  $\mathcal{P}=\{a_1, a_2, \ldots, a_n\}$ is a complete set of inequivalent poles of $f(z)$, and further assume that $f(z)$ satisfies the functional equations
 \begin{equation} \label{kron:eqn11}
 f(z)=f(z+\pi)= e^{2iy} f(z+\pi\tau),
 \end{equation}
 where $e^{2iy}\ne q^n$ for  $n\in \mathbb{Z}$.   Then  we have
\begin{equation}
f(z)=\sum_{k=1}^n \res(f; a_k) K_y(z-a_k|\tau).
\label{kron:eqn12}
\end{equation}
\end{thm}
In the same way we can also prove the following decomposition theorem for theta functions.
\begin{thm}\label{liumainthm:b}
Suppose that $f(z)$ is an meromorphic function of $z$ which has only  simple poles and  $\mathcal{P}=\{a_1, a_2, \ldots, a_n\}$ is a complete set of inequivalent poles of $f(z)$, and further assume that $f(z)$ satisfies the functional equations	
	\begin{equation}\label{kron:eqn13}
	f(z)=-f(z+\pi)=- e^{2iy} f(z+\pi\tau),
	\end{equation}
	where $e^{2iy}\ne -q^{n+1/2}$ for $n\in \mathbb{Z}$. Then we have
	\begin{equation}
	f(z)=\frac{\T_1'(0|\tau)}{\T_3(y|\tau)}\sum_{k=1}^r \res(f; a_k) \frac{\T_3(z+y-a_k|\tau)}{\T_1(z-a_k|\tau)}.
	\label{kron:eqn14}
	\end{equation}
\end{thm}
Theorems~\ref{liumainthm:a} and \ref{liumainthm:b} clearly  include  infinitely  many theta function identities since we can choose $f(z)$ in several ways. If $f(z)$  is chosen judiciously, we may obtain useful results.  For example, we can prove the following two general theta function identities by using these two theorems.
\begin{thm}\label{liumainthm:c}
Suppose that  $F(z)$ and $G(z)$ are two entire functions of $z$ with no common zeros and they satisfy the functional equations
	\begin{equation}	\label{kron:eqn15}
	F(z)=(-1)^n F(z+\pi)=  \pm q^n e^{2inz+2iy} F(z+\pi\tau),
	\end{equation}
	and
	\begin{equation}\label{kron:eqn16}
	G(z)=(-1)^n G(z+\pi)= \pm q^n e^{2inz} G(z+\pi\tau). 
	\end{equation}
We further assume that $G(z)$ has only simple zeros and $\mathcal{P}=\{a_1, a_2, \ldots, a_n\}$ is a complete set of inequivalent zeros of $G(z)$. Then we have
	\begin{equation}\label{kron:eqn17}
	\frac{F(z)\T_1(y|\tau)}{\T_1'(0|\tau)G(z)}
	=\sum_{k=1}^{n}  \frac{F(a_k)\T_1(z+y-a_k|\tau)}
	{G'(a_k)\T_1(z-a_k|\tau)}.
	\end{equation}
\end{thm}	

\begin{thm}\label{liumainthm:d}
	Suppose that  $F(z)$ and $G(z)$ are two entire functions of $z$ with no common zeros and they  satisfy the functional equations
	\begin{equation}	\label{kron:eqn18}
	F(z)=F(z+\pi)=  q^n e^{2inz+2iy}F(z+\pi\tau),
	\end{equation}
	and
	\begin{equation} \label{kron:eqn19}
    G(z)=-G(z+\pi)=-q^n e^{2inz}G(z+\pi\tau)
	\end{equation}
	We further assume that $G(z)$ has only simple zeros and $\mathcal{P}=\{a_1, a_2, \ldots, a_n\}$ is a complete set of inequivalent zeros of $G(z)$. Then we have
	\begin{equation}\label{kron:eqn20}
	\frac{F(z)\T_3(y|\tau)}{\T_1'(0|\tau)G(z)}
	=\sum_{k=1}^{n}  \frac{F(a_k)\T_3(z+y-a_k|\tau)}
	{G'(a_k)\T_1(z-a_k|\tau)}.
	\end{equation}
\end{thm}
The rest of the paper is organized as follows.  In Section~2, we will prove Theorems~\ref{liumainthm:a}--\ref{liumainthm:d}.   Some applications of Theorems~\ref{liumainthm:c} and \ref{liumainthm:d} to theta function identities will be given in Section~3.  For example, we  prove the following remarkable theta function identity by using Theorem~\ref{liumainthm:c}.
\begin{thm}\label{liumainthm:e} If $f(z)$ is an entire function of $z$ which satisfies the functional equations	
	\begin{equation}\label{kron:eqn21}
	f(z)=f(z+\pi)=q^2 e^{8iz+2iy} f(z+\pi\tau),
	\end{equation}	
	then we have
	\begin{align}\label{kron:eqn22}
	\frac{2\T_1(y|\tau) f(z)}{\T_1(2z|\tau)}&=f(0) \frac{\T_1(z+y|\tau)}{\T_1(z|\tau)}
	-f\(\frac{\pi}{2}\) \frac{\T_2(z+y|\tau)}{\T_2(z|\tau)}\\
	&+q^{1/2}f\(\frac{\pi+\pi\tau}{2}\)\frac{e^{iy}\T_3(z+y|\tau)}{\T_3(z|\tau)}
	-q^{1/2}f\(\frac{\pi\tau}{2}\)\frac{e^{iy}\T_4(z+y|\tau)}{\T_4(z|\tau)}.\nonumber
	\end{align}
\end{thm}
In Section~3, as a limiting case of a theta function identity, we also derived the following  amazing trigonometry identity:
\begin{equation}\label{kron:eqn23}
\frac{n \sin (nz+y)}{\sin nz}=
\sum_{k=0}^{n-1}
\frac{\sin (z+y-\frac{k\pi}{n})}{\sin (z-\frac{k\pi}{n})}.
\end{equation}

In Section~4, we will use Theorem~\ref{liumainthm:a} to prove a new addition formula for theta functions and some applications of this addition formula are also discussed. 
In Section~5, we will use Theorem~\ref{liumainthm:a} to give a new derivation of a classical decomposition theorem for elliptic functions and several known results in the theory of elliptic theta functions due to Ramanujan, Weierstrass, Kiepert, Shen  among others are  revisited. 
\section{ Proofs of Theorems~\ref{liumainthm:a}--\ref{liumainthm:d}}
Now we begin to prove Theorem~\ref{liumainthm:a} by using 
some properties of analytic functions.
\begin{proof}
	Using the functional equations satisfied by $\T_1$ in Proposition~\ref{doubleperiods}, we find that $K_y(z|\tau)$ satisfies the functional equations
	\begin{equation}\label{decom:eqn1}
	K_y(z|\tau)=K_y(z+\pi|\tau)=e^{2iy}K_y(z+\pi\tau|\tau).
	\end{equation} 		
	Since $a_k$ is a simple pole of $f(z)$,   from the theory of residue calculus we know that the principal part of $f(z)$ at $a_k$ is
	\begin{equation}\label{decom:eqn2}
	\frac{\res(f; a_k)}{z-a_k}\quad \text{for}\quad  k=1, 2, \ldots n.
	\end{equation}
	From the definition of $\theta_1(z|\tau)$ in \reff{kron:eqn3} and a simple calculation,  we easily find that as $z\to 0,$
	\begin{equation}\label{decom:eqn3}
	\theta_1(z|\tau)=\theta_1'(0|\tau)z+O(z^3).
	\end{equation}
	Using this equation and by a direct computation,  the principal part of $K_y(z|\tau)$ at $z=a_k$ is found
	to be
	\begin{equation}\label{decom:eqn4}
	\frac{1}{z-a_k} \quad \text{for}\quad  k=1, 2, \ldots n.
	\end{equation}
	Hence
	\begin{equation}\label{decom:eqn5}
	r(z):=f(z)-\sum_{k=1}^n \res(f; a_k) K_y(z-a_k|\tau)
	\end{equation}
	has no poles and is holomorphic on $\C,$ and it is easily seen that $r(z)$ satisfies the functional equations
	\begin{equation}\label{decom:eqn6}
	r(z)=r(z+\pi)=e^{2iy}r(z+\pi\tau),
	\end{equation}
	by using the functional equations in \reff{kron:eqn11} and \reff{decom:eqn1}.
	
	From $r(z)=r(z+\pi),$ we know that the Fourier expansion of $r(z)$ has the form
	\[
	r(z)=\sum_{n=-\infty}^\infty a_n e^{2inz}.
	\] 
	Now we replace $z$ by $z+\pi\tau$ in the above equation and appeal to { $r(z)=r(z+\pi\tau)e^{2iy}$}. We deduce that
	\[
	\sum_{n=-\infty}^\infty a_n q^n e^{2inz}=e^{-2iy}\sum_{n=-\infty}^\infty a_n e^{2inz}.
	\] 
	Equating the coefficients, we find that $a_nq^n=a_n e^{-2iy}$ for every integer $n.$ It follows that
	$a_n(q^n-e^{-2iy})=0,$ which gives $a_n=0$ for every integer $n$ since $q^n-e^{-2iy} \not =0$.  Thus we have $r(z)\equiv 0.$ 
	It follows that
	\[
	f(z)=\sum_{k=1}^n \res(f; a_k) K_y(z-a_k|\tau),
	\]
	which is \reff{kron:eqn12} and thus we complete the proof of Theorem~\ref{liumainthm:a}.
\end{proof}
The proof of Theorem~\ref{liumainthm:b} is omitted since it is similar to that of Theorem~\ref{liumainthm:a}.  Next we give the proof of Theorem~\ref{liumainthm:c} by 
using Theorem~\ref{liumainthm:a}. 
\begin{proof}
For the given function  $F(z)$ in Theorem~\ref{liumainthm:c},  it is found that $F(z)/{G(z)}$  is a meromorphic function of $z$ which satisfies the functional equations \reff{kron:eqn11} by using \reff{kron:eqn15}.  Hence in Theorem~\ref{liumainthm:a} we can choose $f(z)=F(z)/G(z).$  It is known that $G(z)$ has only  simple zeros and $ \mathcal{P}=\{a_1, a_2 \ldots, a_n\}$  is a complete set of inequivalent zeros of it. Hence $f(z)$ has only simple poles and $ \mathcal{P}=\{a_1, a_2 \ldots, a_n\}$ is a complete set of inequivalent poles of $f(z)$ since $F(z)$ and $G(z)$ has no common zeros.  Now we begin to calculate the residues at these poles.  

Using L'Hospital's  rule we have the following calculation 
for $k=1, 2, \ldots, n$:
\begin{align*}
\res\(f; a_k\)&=\lim_{z\to a_k} \(z-a_k\) f(z)\\
&=F\(a_k\)\lim_{z\to a_k }  \frac{\(z-a_k\)}{G(z)}\\
&=\frac{F\(a_k \)}{G'(a_k)}.
\end{align*}
Substituting the values of $\res\(f; a_k \)$ into \reff{kron:eqn12} and simplifying we complete the proof of Theorem~\ref{liumainthm:c}.
\end{proof}
Proceeding through the same step as that in the proof of Theorem~\ref{liumainthm:c}, 
we can prove Theorem ~\ref{liumainthm:d}  by using Theorem~\ref{liumainthm:b}.

\section{Some applications of Theorems~\ref{liumainthm:c} and \ref{liumainthm:d}}
To illustrate our method,  we begin this section by proving Theorem~\ref{liumainthm:e}  with Theorem~\ref{liumainthm:c}.
\begin{proof} A simple calculation indicates that  $\T_1(2z|\tau)$ has only simple zeros and a complete set of its inequivalent zeros is given by 
		$\mathcal{P}=\{0, \pi/2, (\pi+\pi\tau)/2, \pi\tau/2\}$.	Let $f(z)$ be the given function in Theorem~\ref{liumainthm:e}. Then we can choose $F(z)=f(z)$ 
		and $G(z)=\theta_1(2z|\tau)$ in Theorem~\ref{liumainthm:c}. By some simple calculations we find that 
		\begin{equation} \label{kj:eqn1}
		G'(0)=2\T_1'(0|\tau),~G'(\pi/2)=-2\T_1'(0|\tau),~G'((\pi+\pi\tau)/2)=2q^{-1/2}\T_1'(0|\tau)
		\end{equation}
		and 
		\[
		G'((\pi\tau)/2)=-2q^{-1/2}\T_1'(0|\tau).
		\]
		Now \reff{kron:eqn17} in Theorem~\ref{liumainthm:c} becomes 
		\begin{align*}
		&\frac{2f(z)\T_1(y|\tau)}{\T_1(2z|\tau)}
		=f(0)\frac{\T_1(z+y|\tau)}{\T_1(z|\tau)}
		-f\(\frac{\pi}{2}\)\frac{\T_1(z+y-\frac{\pi}{2}|\tau)}{\T_1(z-\frac{\pi}{2}|\tau)}\\
		&+q^{1/2}f\(\frac{\pi+\pi \tau}{2}\)\frac{\T_1(z+y-\frac{\pi+\pi \tau}{2}|\tau)}{\T_1(z-\frac{\pi+\pi \tau}{2}|\tau)}
		-q^{1/2}f\(\frac{\pi \tau}{2}\)\frac{\T_1(z+y-\frac{\pi \tau}{2}|\tau)}{\T_1(z-\frac{\pi \tau}{2}|\tau)}.
		\end{align*}
		
Using Proposition~\ref{halfperiods} to simplify the right-hand of the above equation we complete the proof of Theorem~\ref{liumainthm:e}.
\end{proof}	
Letting $y=0$ in Theorem~\ref{liumainthm:e} and noting that $\T_1(0|\tau)=0$, we immediately arrive at the following theorem.
\begin{thm}\label{liumainthm:f} If $f(z)$ is an entire function of $z$ which satisfies the functional equations	
	\begin{equation}\label{kj:eqn2}
	f(z)=f(z+\pi)={ q^2}e^{8iz} f(z+\pi\tau),
	\end{equation}	
	then we have
	\begin{equation}\label{kj:eqn3}
	f(0) 
	-f\(\frac{\pi}{2}\) 
	+q^{1/2}f\(\frac{\pi+\pi\tau}{2}\)
	-q^{1/2}f\(\frac{\pi\tau}{2}\)=0
	\end{equation}
\end{thm}
From Theorem~\ref{liumainthm:f} we can easily find that following proposition 
\cite[Theorem~3]{Liu20001residue}, which has many applications to modular 
equations.
\begin{prop}\label{thetaprop:eqn1} Suppose that $y_1+y_2+y_3+y_4=0$. Then  we have
\begin{equation}\label{kj:eqn4}
\prod_{j=1}^4 \T_1(y_j|\tau)-\prod_{j=1}^4 \T_2(y_j|\tau)+\prod_{j=1}^4 \T_3(y_j|\tau)-\prod_{j=1}^4 \T_4(y_j|\tau)=0.
\end{equation}
\end{prop}
\begin{proof} Keeping $y_1+y_2+y_3+y_4=0$ in mind, in Theorem~\ref{liumainthm:f}  we can take
	\[
	f(z)=\T_1(z+y_1|\tau)\T_1(z+y_2|\tau)\T_1(z+y_3|\tau)\T_1(z+y_4|\tau),
	\]
 to complete the proof of Proposition~\ref{thetaprop:eqn1}.
\end{proof}
By taking $f(z)=\T_1(z+y|\tau)\T_1(3z|3\tau)$ in Theorem~\ref{liumainthm:e} and simplifying we obtain the following new theta function identity:
\begin{align}\label{kj:eqn5}
&\frac{\T_2(y|\tau)\T_2(z+y|\tau)\T_2(0|3\tau)}{\T_2(z|\tau)}
+\frac{\T_4(y|\tau)\T_4(z+y|\tau)\T_4(0|3\tau)}{\T_4(z|\tau)}\\
&=\frac{2\T_1(y|\tau)\T_1(z+y|\tau)\T_1(3z|3\tau)}
{\T_1(2z|2\tau)}+\frac{\T_3(y|\tau)\T_3(z+y|\tau)\T_3(0|3\tau)}{\T_3(z|\tau)}.\nonumber
\end{align}

\begin{thm}\label{appthm:eqn1} For any integer $m$, we have 
\begin{equation}\label{jacobi:eqn1}
e^{2miz}\frac{n\T_1'(0|n\tau)\T_1(nz+m\pi\tau+y|n\tau)\T_1(y|\tau)}{\T_1'(0|\tau)\T_1(y+m\pi\tau|n\tau)\T_1(nz|n\tau)}
=\sum_{k=0}^{n-1}  e^{\frac{2imk\pi }{n}} \frac{\T_1(z+y-\frac{k\pi}{n}|\tau)}
{\T_1(z-\frac{k\pi}{n}|\tau)}.
\end{equation}
\end{thm}
\begin{proof} Using Proposition~\ref{doubleperiods} and simple calculations, we find that  the entire function $e^{2miz}\T_1(nz+m\pi\tau+y|n\tau)$ of $z$ satisfies the functional equations \reff{kron:eqn15} in
Theorem~\ref{liumainthm:c}, and the entire function $\T_1(nz|n\tau)$ of $z$ satisfies the functional equations \reff{kron:eqn16}.  So we can take $F(z)=e^{2miz}\T_1(nz+m\pi\tau+y|n\tau)$ 
and $G(z)=\T_1(nz|n\tau)$ in Theorem~\ref{liumainthm:c}. It is easily known that $\T_1(nz|n\tau)$ has only simple zeros and $\mathcal{P}=\{{k\pi}/{n}: k=0, 1, 2, \ldots, n-1\}$ is a complete set of inequivalent zeros of it. By direct calculations we find that for $k=0, 1, 2, \ldots, n-1,$
\[
G'\(\frac{k\pi}{n}\)=n\T_1'(k\pi|n\tau)=(-1)^k n\T_1'(0|n\tau),
\]
and 
\begin{equation*}
F\left(\frac{k\pi}{n}\right)=(-1)^k e^{\frac{2imk\pi }{n}}\T_1(y+m\pi\tau|n\tau).
\end{equation*} 
Substituting these values into \reff{kron:eqn16} and simplifying, we complete the proof of 	Theorem~\ref{appthm:eqn1}. 	
\end{proof}
Theorem~\ref{jacobi:eqn1} is equivalent to \cite[Eq.(6.7)]{Liu2012IJNT}, but the proof here is very different from that of \cite{Liu2012IJNT}. Next we discuss some special cases of \reff{jacobi:eqn1}. 

Setting $m=0$ in \reff{jacobi:eqn1},  we immediately deduce that
\begin{equation}\label{jacobi:eqn2}
\frac{n\T_1'(0|n\tau)\T_1(nz+y|n\tau)\T_1(y|\tau)}
{\T_1'(0|\tau)\T_1(nz|n\tau)\T_1(y|n\tau)}=
\sum_{k=0}^{n-1} \frac{\T_1(z+y-\frac{k\pi}{n}|\tau)}
{\T_1(z-\frac{k\pi}{n}|\tau)}.
\end{equation}
From the definition of $\T_1(z|\tau)$ in Definition~\ref{thetadefn}, we easily find that near $q=0,$
\begin{equation}\label{jacobi:eqn3}
\T_1(z|\tau)=2q^{1/8}(\sin z)(1+O(q))\quad \text{and}\quad \T_1'(0|\tau)=2q^{1/8}(1+O(q)).
\end{equation}
Using the first equation in \reff{jacobi:eqn3} and a simple calculation, we deduce that as $q \to 0,$
\begin{equation}\label{jacobi:eqn4}
\frac{\T_1(x|\tau)}{\T_1(y|\tau)}\to \frac{\sin x}{\sin y}.
\end{equation}
Substituting \reff{jacobi:eqn3} into \reff{jacobi:eqn2} 
 and then letting $q\to 0$ and using \reff{jacobi:eqn4}, we arrive at the trigonometric identity in \reff{kron:eqn23}.
 
Replacing $y$ by $y+\pi/2$ in \reff{jacobi:eqn1} and noting that $\T_1(z+\pi/2|\tau)=\T_2(z|\tau)$, we conclude that
\begin{equation}\label{jacobi:eqn6}
e^{2miz}\frac{n\T_1'(0|n\tau)\T_2(nz+m\pi\tau+y|n\tau)\T_2(y|\tau)}{\T_1'(0|\tau)\T_2(y+m\pi\tau|n\tau)\T_1(nz|n\tau)}
=\sum_{k=0}^{n-1}  e^{\frac{2imk\pi }{n}} \frac{\T_2(z+y-\frac{k\pi}{n}|\tau)}
{\T_1(z-\frac{k\pi}{n}|\tau)}.
\end{equation}

If $m$ is not a multiple of $n$, we differentiate both sides of \reff{jacobi:eqn1} with respect to $y$ and then set $y=0$ to obtain
\begin{equation} \label{jacobi:eqn7}
e^{2miz} \frac{n\T_1'(0|n\tau)\T_1(nz+m\pi\tau|n\tau)}{\T_1(m\pi\tau|n\tau)\T_1(nz|n\tau)}
=\sum_{k=0}^{n-1}  e^{\frac{2mk\pi i}{n}} \frac{\T_1'(z-\frac{k\pi}{n}|\tau)}
{\T_1(z-\frac{k\pi}{n}|\tau)}.
\end{equation}

If $n$ is an odd integer, we can choose $F(z)=\T_1^n (z+\frac{y}{n}|\tau)$ and 
$G(z)=\T_1(nz|n\tau)$ in Theorem~\ref{liumainthm:c} to obtain the following theorem.
\begin{thm} \label{appthm:eqn2} If $n$ is an odd integer, then,  we have 
	\begin{align} \label{jacobi:eqn8}
	\frac{n \T_1'(0|n\tau) \T_1^n (z+\frac{y}{n}|\tau)\T_1 (y|\tau)}{\T_1'(0|\tau)\T_1(nz|n\tau)}
	=\sum_{k=0}^{n-1} (-1)^k
	\frac{ \T_1 (z+y-\frac{k\pi}{n}|\tau)\T_1^n
		(\frac{y+k\pi}{n}|\tau) }{\T_1(z-\frac{k\pi}{n}|\tau)}.
	\end{align}
\end{thm}
Setting $y=\pi\tau$ in the above equation and noting that $\T_1(\pi\tau|\tau)=0$ we  deduce that for $n$ being odd, 
\begin{equation} \label{jacobi:eqn9}
\sum_{k=0}^{n-1} (-1)^k e^{\frac{2k\pi i}{n}}\T_1^n
\(\frac{k\pi+\pi\tau}{n}\Big|\tau\)=0.
\end{equation}

If $n$ is an odd integer, we can choose 
	$F(z)=e^{2miz}\T_3(nz+m\pi\tau+y|n\tau)$ and $G(z)=\T_1(nz|n\tau)$ in Theorem~\ref{liumainthm:d} to obtain the following theorem.
\begin{thm} \label{appthm:eqn3} If $n$ is an odd integer and $m$ is any integer, then,  we have 
\begin{align} \label{jacobi:eqn10} 
&e^{2miz} \frac{n\T_1'(0|n\tau)\T_3(y|\tau)\T_3(nz+y+m\pi\tau|n\tau)}
{\T_1'(0|\tau)\T_3(y+m\pi\tau|n\tau)\T_1(nz|n\tau)}\\
&=\sum_{k=0}^{n-1} (-1)^k e^{\frac{2km\pi i }{n}}\frac{\T_3(z+y-\frac{k\pi}{n}|\tau)}{\T_1(z-\frac{k\pi}{n}|\tau)} \nonumber.
\end{align}
\end{thm}

Replacing $y$ by $y+\pi/2$ in \reff{jacobi:eqn10} and noting that $\T_3(z+\pi/2|\tau)=\T_4(z|\tau)$, we find that
\begin{align} \label{jacobi:eqn11} 
&e^{2miz} \frac{n\T_1'(0|n\tau)\T_4(y|\tau)\T_4(nz+y+m\pi\tau|n\tau)}
{\T_1'(0|\tau)\T_4(y+m\pi\tau|n\tau)\T_1(nz|n\tau)}\\
&=\sum_{k=0}^{n-1} (-1)^k e^{\frac{2km\pi i }{n}}\frac{\T_4(z+y-\frac{k\pi}{n}|\tau)}{\T_1(z-\frac{k\pi}{n}|\tau)} \nonumber.
\end{align}

If we replace $z$ by $z+(\pi+\pi\tau)/2$ in \reff{jacobi:eqn10}, then  after  tedious calculation, we deduce that
\begin{align} \label{jacobi:eqn12}
&(-1)^{m+\frac{n-1}{2}}ne^{2miz} \frac{\T_1'(0|n\tau)\T_3(y|\tau)\T_1(nz+y+m\pi\tau|n\tau)}
{\T_1'(0|\tau)\T_3(y+m\pi\tau|n\tau)\T_3(nz|n\tau)}\\
&=\sum_{k=0}^{n-1} (-1)^k e^{\frac{2ikm\pi }{n}}\frac{\T_1(z+y-\frac{k\pi}{n}|\tau)}{\T_3(z-\frac{k\pi}{n}|\tau)} \nonumber.
\end{align}

\begin{thm}\label{appthm:eqn4}
	If $n$ is an odd integer, then, we have
	\begin{equation}\label{jacobi:eqn13}
		\frac{\T_1'(0|\frac{\tau}{n})\T_1(z+\frac{y}{n}|\frac{\tau}{n})\T_1(y|\tau)}
		{\T_1'(0|\tau)\T_1(z|\frac{\tau}{n})\T_1(\frac{y}{n}|\frac{\tau}{n})}=
		\sum_{k=0}^{n-1} e^{-\frac{2kiy}{n}} \frac{\T_1(z+y-\frac{k\pi \tau}{n}|\tau)}
		{\T_1(z-\frac{k\pi\tau}{n}|\tau)}.
	\end{equation}
	\end{thm}
\begin{proof} 
For odd integer $n$,  using Proposition~\ref{doubleperiods} and simple calculations, we find that  the entire function $\T_1(z+\frac{y}{n}|\frac{\tau}{n})$ of $z$ satisfies the functional equations \reff{kron:eqn15} in
Theorem~\ref{liumainthm:c}, and the entire function $\T_1(z|\frac{\tau}{n})$ of $z$ satisfies the functional equations \reff{kron:eqn16} in
Theorem~\ref{liumainthm:c}.  So we can take $F(z)=\T_1(z+\frac{y}{n}|\frac{\tau}{n})$ 
and $G(z)=\T_1(z|\frac{\tau}{n})$ in Theorem~\ref{liumainthm:c}. It is easily known that { $\T_1(z|\frac{\tau}{n})$} has only simple zeros and $\mathcal{P}=\{{k\pi\tau}/{n}: k=0, 1, 2, \ldots, n-1\}$ is a complete set of inequivalent zeros of $G(z)=\T_1(z|\frac{\tau}{n})$. 
By direct calculations we find that for $k=0, 1, 2, \ldots, n-1,$
\[
G'\(\frac{k\pi \tau}{n}\)=\T_1'\(\frac{k\pi\tau}{n}\Big|\frac{\tau}{n}\)=(-1)^k q^{-\frac{k^2}{2n}}\T_1'\(0\Big|\frac{\tau}{n}\),
\]
and 
\begin{equation*}
F\left(\frac{k\pi\tau}{n}\right)=(-1)^k q^{-\frac{k^2}{2n}} e^{-\frac{2kiy}{n}}\T_1\(\frac{y}{n}\Big|\frac{\tau}{n}\).
\end{equation*}	
Substituting these values into \reff{kron:eqn16} and simplifying, we complete the proof of Theorem~\ref{appthm:eqn4}.	
\end{proof}	

For odd integer $n$,  choosing $F(z)=\T_1^n(z+\frac{y}{n}|\tau)$ and 
$G(z)=\T_1(z|\frac{\tau}{n})$ in Theorem~\ref{liumainthm:c},  we are led to
the following theorem.
\begin{thm}\label{appthm:eqn5}
	If $n$ is an odd integer, then, we have
	\begin{equation}\label{jacobi:eqn14}
	\frac{\T_1'(0|\frac{\tau}{n})\T_1 ^n(z+\frac{y}{n}|{\tau})\T_1(y|\tau)}
	{\T_1'(0|\tau)\T_1(z|\frac{\tau}{n})\T_1(\frac{y}{n}|\frac{\tau}{n})}=
	\sum_{k=0}^{n-1} (-1)^k \T_1^n\(\frac{y+k\pi\tau}{n}\Big|\tau\)\frac{\T_1(z+y-\frac{k\pi \tau}{n}|\tau)}
	{\T_1(z-\frac{k\pi\tau}{n}|\tau)}q^{\frac{k^2}{2n}}.
	\end{equation}
\end{thm}
Setting $y=\pi$ in \reff{jacobi:eqn14} and noting that $\T_1(\pi|\tau)=0$, we immediately  conclude that for $n$ being odd, 
\begin{equation}\label{jacobi:eqn15}
\sum_{k=0}^{n-1} (-1)^k \T_1^n\(\frac{\pi+k\pi\tau}{n}\Big|\tau\)q^{\frac{k^2}{2n}}=0.
\end{equation}
\section{A new addition formula for theta functions}

In this section we will use Theorem~\ref{liumainthm:a} to prove the following addition formula for theta functions.
\begin{thm} \label{addthm} Let $F(z)$ and $G(z)$ be two entire functions of $z$ that satisfy the functional equations
	\begin{equation}
	\begin{split}
	F(z)=F(z+\pi)&= q e^{(2\alpha+4)iz}F(z+\pi\tau),\\
	G(z)=G(z+\pi)&=q e^{(2\alpha+4)iz} G(z+\pi\tau).
	\end{split}
	\label{add:eqn1}
	\end{equation}
	Then there exists a constant $C$ independent of $x$ and $y$ such that
	\begin{align}
	\frac{F(x)}{G(x)}-\frac{F(y)}{G(y)}=C\frac{\T_1(x+y+\alpha|\tau)\T_1(x-y|\tau)}{G(x)G(y)}.
	\label{add:eqn2}
	\end{align}
\end{thm}
\begin{proof} Let $f(z)=F(z)/{\T_1(z-x|\tau)\T_1(z+x|\tau)}$. Then it is easy to verify that $f(z+\pi)=f(z)$ and $f(z+\pi\tau)=e^{-2\alpha i} f(z).$ Using Theorem~\ref{liumainthm:a} we have the decomposition formula
\begin{align*}
\frac{F(z)\T_1(2x|\tau)\T_1(\alpha|\tau)}{\T_1(z-x|\tau)\T_1(z+x|\tau)}=F(x)\frac{\T_1(z+\alpha-x)}{\T_1(z-x|\tau)}
-F(-x)\frac{\T_1(z+\alpha+x)}{\T_1(z+x|\tau)}.
\end{align*}	
In the same way we also have the decomposition formula	
\begin{align*}
\frac{G(z)\T_1(2x|\tau)\T_1(\alpha|\tau)}{\T_1(z-x|\tau)\T_1(z+x|\tau)}=G(x)\frac{\T_1(z+\alpha-x)}{\T_1(z-x|\tau)}
-G(-x)\frac{\T_1(z+\alpha+x)}{\T_1(z+x|\tau)}.
\end{align*}
Eliminating ${\T_1(z+\alpha-x)}/{\T_1(z-x|\tau)}$ from the above two equations and then replacing $z$ by $y$, we find that
\begin{align}\label{add:eqn3}
F(x)G(y)-G(x)F(y)=C(x)\T_1(x+y+\alpha|\tau)\T_1(x-y|\tau),
\end{align}
where 
\begin{equation}\label{add:eqn4}
C(x)=\frac{F(x)G(-x)-G(x)F(-x)}{\T_1(2x|\tau)\T_1(\alpha|q)}.
\end{equation}
Interchanging $x$ and $y$ in \reff{add:eqn3} we immediately deduce that
\begin{align}\label{add:eqn5}
F(x)G(y)-G(x)F(y)=C(y)\T_1(x+y+\alpha|\tau)\T_1(x-y|\tau).
\end{align}
Comparing \reff{add:eqn3} and \reff{add:eqn5}, we find that $C(x)=C(y)$ which shows that $C(x)$ is independent of $x$, and so it must be a constant, say $C$. This completes the proof of Theorem~\ref{addthm}.
\end{proof}
It is well-known that the Weierstrass elliptic function $\wp(z|\tau)$ attached to the periodic lattice $\Lambda$ is defined by
\[
\wp(z|\tau)=\frac{1}{z^2}+\sum_{{\omega \in \Lambda}\atop{\omega \not=0}}
\( \frac{1}{(z-\omega)^2}-\frac{1}{\omega^2}\),
\]
which has primitive periods $\pi$ and $\pi\tau$. Also it has only one inequivalent pole at $z=0$, of order two.

Next we will give some applications of Theorem~\ref{addthm}.
\begin{thm}\label{addthm:app1} Let $\wp(z|\tau)$ be the Weierstrass elliptic function. Then we have
\begin{align}\label{add:eqn6}
\frac{\T_1^2(x|\tau)\wp(x|\tau)}{\T_1(x-u|\tau)\T_1(x+u|\tau)}-\frac{\T_1^2(y|\tau)\wp(y|\tau)}{\T_1(y-u|\tau)\T_1(y+u|\tau)}\\
=-\frac{\T_1^2(u|\tau)\T_1(x+y|\tau)\T_1(x-y|\tau)\wp(u|\tau)}{\T_1(x-u|\tau)\T_1(x+u|\tau)\T_1(y-u|\tau)\T_1(y+u|\tau)}.\nonumber
\end{align}		
\end{thm}
\begin{proof} It is easy to verify that two entire functions $\T_1^2(z|\tau)\wp(z|\tau)$ and $\T_1(z-u|\tau)\T_1(z+u|\tau)$ of $z$ both satisfy the functional equations in \reff{add:eqn1} with $\alpha=0$. Hence we can take
\[
F(z)=\T_1^2(z|\tau)\wp(z|\tau)\quad  \text{and} \quad { G(z)=\T_1(z-u|\tau)\T_1(z+u|\tau)}
\]	
in Theorem~\ref{addthm} to deduce that
\begin{align}\label{add:eqn7}
\frac{\T_1^2(x|\tau)\wp(x|\tau)}{\T_1(x-u|\tau)\T_1(x+u|\tau)}-\frac{\T_1^2(y|\tau)\wp(y|\tau)}{\T_1(y-u|\tau)\T_1(y+u|\tau)}\\
=\frac{C\T_1(x+y|\tau)\T_1(x-y|\tau)}{\T_1(x-u|\tau)\T_1(x+u|\tau)\T_1(y-u|\tau)\T_1(y+u|\tau)}.\nonumber
\end{align}
Multiplying both sides of the above equation by $x-u$ and then letting $x\to u$, we conclude that
\[
C=-\T_1^2(u|\tau)\wp(u|\tau).
\]
Substituting this value into the right-hand side of \reff{add:eqn7}	we complete the proof of Theorem~\ref{addthm:app1}.	
\end{proof}
It is well-known that the Weierstrass elliptic function $\wp(z|\tau)$ has the Laurent expansion near $z=0,$
\begin{equation}\label{add:eqn8}
\wp(z|\tau)=\frac{1}{z^2}+O(z^2).
\end{equation}
Using this fact, we find that $\T_1^2(u|\tau)\wp(u|\tau) \to \T_1'(0|\tau)^2$ as $u\to 0$. Hence letting $u\to 0$ in \reff{add:eqn6}, we immediately arrive at the 
addition formula for the Weierstrass sigma-function:
\begin{equation}\label{add:eqn9}
\wp(x|\tau)-\wp(y|\tau)=-\T_1'(0|\tau)^2 \frac{\T_1(x+y|\tau)\T_1(x-y|\tau)}
{\T_1^2(x|\tau)\T_1^2(y|\tau)}.
\end{equation}
\begin{thm}\label{addthm:app2} We have 
\begin{align}\label{add:eqn10}
&\T_1(x+u+w|\tau)\T_1(x-u|\tau)\T_1(y+v+w|\tau)\T_1(y-v|\tau)\\
&-\T_1(y+u+w|\tau)\T_1(y-u|\tau)\T_1(x+v+w|\tau)\T_1(x-v|\tau)\nonumber\\
&=\T_1(x-y|\tau)\T_1(x+y+w|\tau)\T_1(u+v+w|\tau)\T_1(u-v|\tau).\nonumber
\end{align}		
\end{thm}	
\begin{proof}
It is easy to verify that  $\T_1(z+u+w|\tau)\T_1(z-u|\tau)$ and $\T_1(z+v+w|\tau)\T_1(z-v|\tau)$ both satisfy the functional equations in \reff{add:eqn1}  with $\alpha=w$. So in Theorem~\ref{addthm} we can take 
\[
F(z)=\T_1(z+u+w|\tau)\T_1(z-u|\tau)\quad  \text{and}\quad  G(z)=\T_1(z+v+w|\tau)\T_1(z-v|\tau).
\]
After a simple calculation, we complete the proof of Theorem~\ref{addthm:app2}.
\end{proof}
When $w=0$, Theorem~\ref{addthm:app2} reduces to the Weierstrass three-term theta 
function identity \cite[Eq.(1.1)]{Koornwinder}, \cite[Theorem~7]{Liu2007ADV}, \cite[p.451, Example~5]{W-W}.

Replacing $u$ by $u+(\pi+\pi\tau)/2$ and $v$ by $v+\pi\tau/2$ in \reff{add:eqn10}
and simplifying we conclude that
	\begin{align}\label{add:eqn11}
	&\T_3(y+u+w|\tau)\T_3(y-u|\tau)\T_4(x+v+w|\tau)\T_4(x-v|\tau)\\
	&-\T_3(x+u+w|\tau)\T_3(x-u|\tau)\T_4(y+v+w|\tau)\T_4(y-v|\tau)\nonumber\\
	&=\T_1(x-y|\tau)\T_1(x+y+w|\tau)\T_2(u+v+w|\tau)\T_2(u-v|\tau).\nonumber
	\end{align}			
When $w=0$, the above equation reduces to \cite[Theorem~1.3]{Liu2009}, which 
includes many well-known addition formulas for the Jacobi theta
functions as special cases.

Using Theorem~\ref{addthm} we can also prove the following theta function identity \cite[Theorem~3]{Liu2007ADV}. This identity is equivalent to Winquist's identity 
\cite{Winquist1969}, which was used by him to give a simple proof of Ramanujan’s partition congruence for the modulus $11$, $p(11n+6)\equiv 0 \pmod {11}$, where $p(n)$ denotes the number of unrestricted partitions of the positive
integer $n$.
\begin{thm}\label{addthm:app3} We have
\begin{align*}
&q^{1/4} (q; q)^2_\infty \T_1(3y|3\tau) \(e^{2ix}\T_1(3x+\pi\tau|3\tau)+e^{-2ix}\T_1(3x-\pi\tau|3\tau)\)\\
&-q^{1/4} (q; q)^2_\infty \T_1(3x|3\tau) \(e^{2iy}\T_1(3y+\pi\tau|3\tau)+e^{-2iy}\T_1(3y-\pi\tau|3\tau)\)\\
&=\T_1(x|\tau)\T_1(y|\tau)\T_1(x+y|\tau)\T_1(x-y|\tau).
\end{align*}	
\end{thm}	
\begin{proof}
	In Theorem~\ref{addthm}	by choosing $F(z)=\T_1(3z|3\tau)/{\T_1(z|\tau)}$ and 
	\[
    G(z)=\frac{e^{2iz}\T_1(3z+\pi\tau|3\tau)+e^{-2iz}\T_1(3z-\pi\tau|3\tau)}{\T_1(z|\tau)}
	\]
 and making a simple calculation, we can complete the proof of 	Theorem~\ref{addthm:app3}.
\end{proof}	
\section{A classical decomposition formula for elliptic functions}
In this section we will use the same method  as that of proving Theorem~\ref{liumainthm:a} to give a  proof of the following classical decomposition formula for elliptic functions (see, for example \cite[Eq.(11)]{Basoco1931}).
\begin{thm} \label{decomthm} Suppose that $f(z)$ is an elliptic function with periods $\pi$ and $\pi\tau$ which has only simple poles,  and $\mathcal{P}=\{a_1, a_2, \ldots, a_n\}$  is complete set of inequivalent poles.  Then for some constant $C$, we have
	\begin{equation}
	f(z)=C+\sum_{k=1}^n \res(f; a_k) \frac{\theta'_1(z-a_k|\tau)}{\theta_1(z-a_k|\tau)}.
	\label{ell:eqn1}
	\end{equation}
\end{thm}
\begin{proof}
{ It is easily known that we can assume the principal part of $f(z)$ at $z_k$ is
\[
\frac{\res(f; a_k)}{z-a_k}\quad \text{for}\quad  k=1, 2, \ldots, n.
\]
By a  direct computation the principal part of ${\theta'_1(z-a_k|\tau)}/{\theta_1(z-a_k|\tau)}$ at $z=a_k$ is found
to be
\[
\frac{1}{z-a_k} \quad \text{for}\quad  k=1, 2, \ldots, n.
\]
Hence
\[
f(z)-\sum_{k=1}^n  \res(f; a_k) \frac{\theta'_1(z-a_k|\tau)}{\theta_1(z-a_k|\tau)}
\]
has no poles and is holomorphic on the whole complex plane.  Its derivative is an elliptic function with no poles since the second order logarithmic derivative of $\T_1$ is elliptic.  So the above function must be constant, say $C.$
We completes the proof of Theorem~\ref{decomthm}.}	
\end{proof}
\section{Some applications of Theorem~\ref{decomthm}}
There are many wonderful applications of Theorem~\ref{decomthm}, and we will give some examples in this section.
\subsection{The Kiepert quintuple product identity }
\begin{thm} \label{decomthm:app1}If $f(z)$ is an entire function of $z$ which satisfies the functional equations
	\begin{equation}\label{Kiepert:eqn1}
	f(z)=f(z+\pi)=q^2 e^{8iz} f(z+\pi\tau),
	\end{equation}
	then we have 
	\[
	f(z)-f(-z)=\frac{f'(0)}{\T_1'(0|\tau)}\T_1(2z|\tau).
	\]
\end{thm}
\begin{proof} By a direct computation we find that $\T_1(2z|\tau)$ has only simple zeros and a complete set of its inequivalent zeros is given by 
$\mathcal{P}=\{0, \pi/2, (\pi+\pi\tau)/2, \pi\tau/2\}$.	Let $f(z)$ be the given function in Theorem~\ref{decomthm:app1}. Then $f(z)/{\T_1(2z|\tau)}$  is an elliptic function  which has only simple poles and a complete set of its inequivalent poles is given by 
$\mathcal{P}=\{0, \pi/2, (\pi+\pi\tau)/2, \pi\tau/2\}$. Using Theorem~\ref{decomthm} we can easily find that
	\begin{align}\label{Kiepert:eqn2}
	\frac{f(z)}{\T_1(2z|\tau)}&=c+\frac{f(0)\T_1'(z|\tau)}{2\T_1'(0|\tau)\T_1(z|\tau)}
	-\frac{f(\frac{\pi}{2})\T_2'(z|\tau)}{2\T_1'(0|\tau)\T_2(z|\tau)}\\
	&\qquad+\frac{q^{1/2}f(\frac{\pi+\pi\tau}{2})\T_3'(z|\tau)}{2\T_1'(0|\tau)\T_3(z|\tau)}
	-\frac{q^{1/2} f(\frac{\pi\tau}{2})\T_4'(z|\tau)}{2\T_1'(0|\tau)\T_4(z|\tau)}.\nonumber
	\end{align}
	Replacing $z$ by $-z$ and noting that $\T_1(z|\tau)$ and $\T_k'(z|\tau)/\T_k(z|\tau)$ for $k\in\{1, 2, 3, 4\}$ are odd function of $z$, we find that
		\begin{align}\label{Kiepert:eqn3}
	\frac{f(-z)}{\T_1(2z|\tau)}&=-c+\frac{f(0)\T_1'(z|\tau)}{2\T_1'(0|\tau)\T_1(z|\tau)}
	-\frac{f(\frac{\pi}{2})\T_2'(z|\tau)}{2\T_1'(0|\tau)\T_2(z|\tau)}\\
	&\qquad+\frac{q^{1/2}f(\frac{\pi+\pi\tau}{2})\T_3'(z|\tau)}{2\T_1'(0|\tau)\T_3(z|\tau)}
	-\frac{q^{1/2} f(\frac{\pi\tau}{2})\T_4'(z|\tau)}{2\T_1'(0|\tau)\T_4(z|\tau)}.\nonumber
	\end{align}
	Taking the difference of the above two equation, we immediately deduce that
	\[
	\frac{f(z)-f(-z)}{\T_1(2z|\tau)}=2c.
	\]
	Letting $z\to 0$ in the above equation, we find that $2c=f'(0)/\T_1'(0|\tau)$. This  complete the proof of Theorem~\ref{decomthm:app1}.
\end{proof}
By taking $f(z)=e^{iz}\T_1(z|\tau)\T_4\(3z+\frac{\pi \tau}{2}|3\tau\)$ in Theorem~\ref{decomthm:app1} and noting that
\begin{equation*}
f'(0)= \T_1'(0|\tau)\T_4\(\frac{\pi \tau}{2}|3\tau\)
=\T_1'(0|\tau)\prod_{n=1}^\infty (1-q^n),
\end{equation*}
and 
\begin{align*}
e^{iz}\T_4\(3z+\frac{\pi \tau}{2}|3\tau\)-e^{-iz}\T_4\(3z-\frac{\pi \tau}{2}|3\tau\)
=2\sum_{n=-\infty}^\infty (-1)^n q^{n(3n+1)/2} \cos (6n+1)z,
\end{align*}
we deduce that the quintuple product identity
\begin{align}\label{Kiepert:eqn4}
2\sum_{n=-\infty}^\infty (-1)^n q^{n(3n+1)/2} \cos (6n+1)z
=\(\prod_{n=1}^\infty (1-q^n) \) \frac{\T_1(2z|\tau)}{\T_1(z|\tau)}.
\end{align}
This identity was first discovered by Kiepert \cite[p.213, Eq.(27)]{Kiepert1879} in 1879 and then rediscovered several times by others.

For many other applications of Theorem~\ref{decomthm:app1}, please refer to Liu \cite{Liu2005ADV}.
\subsection{Two Eisenstein series identities due to Ramanujan}
The trigonometric series expansions for the logarithmic derivatives of $\T_1(z|\tau)$ and $\T_4(z|\tau)$  are given by
\begin{equation}\label{Kiepert:eqn5}
\frac{\T_1'(z|\tau)}{\T_1(z|\tau)}=\cot z+4\sum_{n=1}^\infty \frac{q^n}{1-q^n} \sin 2nz, 
\end{equation}
and {
\begin{equation}\label{Kiepert:eqn6}
\frac{\T_4'(z|\tau)}{\T_4(z|\tau)}=4\sum_{n=1}^\infty \frac{q^{n/2}}{1-q^n} \sin 2nz. 
\end{equation}}
\begin{thm}\label{decomthm:app2}Let $\(a/{p}\)$ be the Legendre symbol modulo $p$ and $\eta(\tau)$ be the Dedekind eta function defined as in \reff{kron:eqn6}. Then we have
\begin{equation}\label{Kiepert:eqn7}	
\frac{\sin z \sin 2z}{\sin 5z}-\sum_{n=1}^\infty 
\(\frac{n}{5}\) \frac{q^n}{1-q^n} \sin 2nz
=\frac{ \eta^2(5\tau)\T_1(z|\tau)\T_1(2z|\tau)}{2\eta(\tau)\T_1(5z|5\tau)},	
\end{equation}	
and 
\begin{align}\label{Kiepert:eqn8}
\sum_{n=1}^\infty \frac{(q^n-q^{2n}-q^{3n}+q^{4n})}{1-q^{5n}} \sin 2nz
=\frac{\eta^2(\tau)\T_1(z|5\tau)\T_1(2z|5\tau)}{2\eta(5\tau)\T_1(z|\tau)}.
\end{align}
\end{thm}
The identity in \reff{Kiepert:eqn7} was first established by the author in  \cite[Proposition~5.1]{Liu2012JNT} and \reff{Kiepert:eqn8} is implied in
\cite[Theorem~1]{Liu2007JRMS}. Now we will show that these two Lambert series identities can be derived very  naturally from Theorem~\ref{decomthm}.

\begin{proof} It is easily seen that theta function ${\T_1(5z|5\tau)}$ has only simple zeros and a complete set of its equivalent zeros is given by 
$\mathcal{P}_1=\{0, \pi/5, -\pi/5, 2\pi/5, -2\pi/5\}$. It follows that the elliptic function $\T_1(z|\tau)\T_1(2z|\tau)/{\T_1(5z|5\tau)}$ has only simple poles and a complete set of inequivalent poles is given by $\mathcal{P}=\{\pi/5, -\pi/5, 2\pi/5, -2\pi/5\}$ since $0$ is a zero of $\T_1(z|\tau)$. Using Theorem~\ref{decomthm} and some simple calculations we have the  decomposition formula
\begin{align*}
&\frac{2\sqrt{5} \eta^2(5\tau)\T_1(z|\tau)\T_1(2z|\tau)}{\eta(\tau)\T_1(5z|5\tau)}\\
&=-\frac{\T_1'(z-\frac{\pi}{5}\Big|\tau)}{\T_1(z-\frac{\pi}{5}\Big|\tau)}
-\frac{\T_1'(z+\frac{\pi}{5}\Big|\tau)}{\T_1(z+\frac{\pi}{5}\Big|\tau)}
+\frac{\T_1'(z-\frac{2\pi}{5}\Big|\tau)}{\T_1(z-\frac{2\pi}{5}\Big|\tau)}
+\frac{\T_1'(z+\frac{2\pi}{5}\Big|\tau)}{\T_1(z+\frac{2\pi}{5}\Big|\tau)}.
\end{align*}
Using \reff{Kiepert:eqn5} to  simplify the right-hand the above equation we find that
\begin{align*}
-\frac{\T_1'(z-\frac{\pi}{5}\Big|\tau)}{\T_1(z-\frac{\pi}{5}\Big|\tau)}
-\frac{\T_1'(z+\frac{\pi}{5}\Big|\tau)}{\T_1(z+\frac{\pi}{5}\Big|\tau)}
+\frac{\T_1'(z-\frac{2\pi}{5}\Big|\tau)}{\T_1(z-\frac{2\pi}{5}\Big|\tau)}
+\frac{\T_1'(z+\frac{2\pi}{5}\Big|\tau)}{\T_1(z+\frac{2\pi}{5}\Big|\tau)}\\
={4\sqrt{5}}\(\frac{\sin z \sin 2z}{\sin 5z}-\sum_{n=1}^\infty 
\(\frac{n}{5}\) \frac{q^n}{1-q^n} \sin 2nz\).
\end{align*}
Combining the above equation we immediately arrive at \reff{Kiepert:eqn7}.

Following the similar steps  as above, we can find that
\begin{align*}
&\frac{2\eta^2(\tau)\T_1(z|5\tau)\T_1(2z|5\tau)}{\eta(5\tau)\T_1(z|\tau)}\\
&= \frac{\T_1'(z+\pi\tau|5\tau)}{\T_1(z+\pi\tau|5\tau)}+\frac{\T_1'(z-\pi\tau|5\tau)}{\T_1(z-\pi\tau|5\tau)}
-\frac{\T_1'(z+2\pi\tau|5\tau)}{\T_1(z+2\pi\tau|5\tau)}-\frac{\T_1'(z-2\pi\tau|5\tau)}{\T_1(z-2\pi\tau|5\tau)}\\
&=\frac{\T_4'(z+\frac{3\pi\tau}{2}|5\tau)}{\T_4(z+\frac{3\pi\tau}{2}|5\tau)}+\frac{\T_4'(z-\frac{3\pi\tau}{2}|5\tau)}{\T_4(z-\frac{3\pi\tau}{2}|5\tau)}
-\frac{\T_4'(z+\frac{\pi\tau}{2}|5\tau)}{\T_4(z+\frac{\pi\tau}{2}|5\tau)}-\frac{\T_4'(z-\frac{\pi\tau}{2}|5\tau)}{\T_4(z-\frac{\pi\tau}{2}|5\tau)}\\
&=4\sum_{n=1}^\infty \frac{(q^n-q^{2n}-q^{3n}+q^{4n})}{1-q^{5n}} \sin 2nz,
\end{align*}
which indicates that  \reff{Kiepert:eqn8} holds. Thus we complete the  proof of Theorem~\ref{decomthm:app2}.
\end{proof}	
In his manuscript on the partition and tau functions, first published with "Lost Notebook" \cite[pp.139--140]{Ramanujan1988}, Ramanujan states without proof that 
\begin{equation}\label{Kiepert:eqn9}
1-5\sum_{n=1}^\infty \(\frac{n}{5}\)\frac{nq^n}{1-q^n}=\frac{\eta^5(\tau)}{\eta(5\tau)},
\end{equation}
and 
\begin{equation}\label{Kiepert:eqn10}
\sum_{n=1}^\infty \(\frac{n}{5}\) \frac{q^n}{(1-q^n)^2}=\frac{\eta^5(5\tau)}{\eta(\tau)}.
\end{equation}
Now we will show that the above two Ramanujan's identities can be derived from Theorem~\ref{decomthm:app2} easily.
\begin{proof}
Dividing both sides of  \reff{Kiepert:eqn7} by $z$ and then letting $z\to 0$, we arrive at Ramanujan's identity in \reff{Kiepert:eqn9}.

Dividing both sides of  \reff{Kiepert:eqn8} by $z$,  and then letting $z\to 0$ 
and noting that
\[
\sum_{n=1}^\infty \frac{n(q^n-q^{2n}-q^{3n}+q^{4n})}{(1-q^{5n})}
=\sum_{n=1}^\infty \(\frac{n}{5}\) \frac{q^n}{(1-q^n)^2},
\]
we obtain Ramanujan's identity in \reff{Kiepert:eqn10}.
\end{proof}
Hence  \reff{Kiepert:eqn7} is a parameterization of Ramanujan's identity in \reff{Kiepert:eqn9} and \reff{Kiepert:eqn8} is a parameterization of Ramanujan's identity in \reff{Kiepert:eqn10}. 

Let's give some other applications of Theorem~\ref{decomthm:app2}.

Replacing $z$ by $z+\pi/2$ in \reff{Kiepert:eqn7} we easily find the following Lambert series identity:
\begin{equation}\label{Kiepert:eqn11}	
\frac{\cos z \sin 2z}{\cos 5z}+\sum_{n=1}^\infty (-1)^n
\(\frac{n}{5}\) \frac{q^n}{1-q^n} \sin 2nz
=\frac{ \eta^2(5\tau)\T_2(z|\tau)\T_1(2z|\tau)}{2\eta(\tau)\T_2(5z|5\tau)}.
\end{equation}

Dividing both sides of the above equation by $z$ and then letting $z\to 0$, we arrive at the following identity due to Shen \cite[Eq.(3.2)]{Shen1994TAMS}:
\begin{equation}\label{Kiepert:eqn12}
1+\sum_{n=1}^\infty (-1)^n \(\frac{n}{5}\) \frac{nq^n}{1-q^n}
=\frac{\eta(\tau)\eta^2(2\tau)\eta^3(5\tau)}{\eta^2(10\tau)}.
\end{equation}

Setting $z=\pi/3$ in \reff{Kiepert:eqn11} and using the infinite product representation of $\T_2$ to simplify the resulting equation, we conclude that
\begin{equation}\label{Kiepert:eqn13}
1+\sum_{n=1}^\infty (-1)^n\(\frac{n}{15}\) \frac{nq^n}{1-q^n}=
\frac{\eta(\tau)\eta(6\tau)\eta(10\tau)\eta(15\tau)}{\eta(2\tau)\eta(30\tau)},
\end{equation}
where $(n/{15})$ is the Jacobi symbol.

Taking $z=\pi/4$ in \reff{Kiepert:eqn11} and simplifying, we conclude that
\begin{equation}\label{Kiepert:eqn14}
1+\sum_{n=0}^\infty (-1)^n \(\frac{2n+1}{5}\) \frac{{(2n+1)}q^{2n+1}}{1-q^{2n+1}}=
\frac{\eta(2\tau)\eta(4\tau)\eta(5\tau)\eta(10\tau)}{\eta(\tau)\eta(20\tau)}.
\end{equation}

Replacing $z$ by $z+\pi/2$ in \reff{Kiepert:eqn8} we easily find the following Lambert series identity:
\begin{equation}\label{Kiepert:eqn15}	
\sum_{n=1}^\infty (-1)^{n-1}
 \frac{(q^n-q^{2n}-q^{3n}+q^{4n})}{1-q^{5n}} \sin 2nz
=\frac{ \eta^2(\tau)\T_2(z|5\tau)\T_1(2z|5\tau)}{2\eta(5\tau)\T_2(z|\tau)}.
\end{equation}

Setting $z=\pi/3$ in the above equation and simplifying we find that
\begin{equation}\label{Kiepert:eqn16}
\sum_{n=1}^\infty (-1)^{n-1} \(\frac{n}{3}\) \frac{(q^n-q^{2n}-q^{3n}+q^{4n})}
{1-q^{5n}} =\frac{\eta(2\tau)\eta(3\tau)\eta(5\tau)\eta(30\tau)}{\eta(6\tau)\eta(10\tau)}.
\end{equation}

Dividing both sides of \reff{Kiepert:eqn1} by $z$ and then letting $z\to 0$, we deduce that
\begin{equation}\label{Kiepert:eqn17}	
\sum_{n=1}^\infty (-1)^{n-1}
 \frac{n(q^n-q^{2n}-q^{3n}+q^{4n})}{1-q^{5n}} 
=\frac{ \eta^3(2\tau)\eta(5\tau)\eta^2(10\tau)}{\eta^2(2\tau)}.
\end{equation}

Using Theorem~\ref{decomthm} we can also prove the following Lambert series identity.
\begin{thm}\label{decomthm:app3} We have
\begin{equation}\label{Kiepert:eqn18}
\sum_{n=1}^\infty \(\frac{n}{5}\) \frac{q^n}{1-q^{2n}}\sin 2nz 
=\frac{\eta^2(10\tau)\T_4(z|2\tau)\T_1(2z|2\tau)}{2\eta(2\tau)\T_4(5z|10\tau)}.
\end{equation}
\end{thm}

Dividing both sides of \reff{Kiepert:eqn18} by $z$ and then letting $z\to 0$, we obtain the following identities of Shen \cite[Eq.(3.32)]{Shen1994TAMS}:
\begin{equation}\label{Kiepert:eqn19}
\sum_{n=1}^\infty \(\frac{n}{5}\) \frac{nq^n}{1-q^{2n}}
=\frac{\eta^2(\tau)\eta(2\tau)\eta^3(10\tau)}{\eta^2(5\tau)}.
\end{equation}

Putting $z=\pi/3$ in \reff{Kiepert:eqn18} and simplifying, we are led to the identity
\begin{equation}\label{Kiepert:eqn20}
\sum_{n=1}^\infty \(\frac{n}{15}\) \frac{nq^n}{1-q^{2n}}
=\frac{\eta(2\tau)\eta(3\tau)\eta(5\tau)\eta(30\tau)}{\eta (\tau)\eta(15\tau)}.
\end{equation}

Putting $z=\pi/4$ in \reff{Kiepert:eqn18} and simplifying, we are led to the identity
\begin{equation}\label{Kiepert:eqn21}
\sum_{n=0}^\infty (-1)^n \(\frac{2n+1}{5}\) \frac{q^{2n+1}}{1-q^{4n+2}}
=\frac{\eta^4(4\tau)\eta^2(10\tau)\eta(40\tau)}
{\eta^2(2\tau)\eta(8\tau)\eta^2(20\tau)}.
\end{equation}
\subsection{An Eisenstein series identity due to Carlitz}
\begin{thm}\label{Carlitzthm} The following identity involving theta functions and Lambert series holds:
\begin{equation}\label{Kiepert:eqn22}
\frac{\sin^3 z}{3\sin 3z}-\sum_{n=1}^\infty \(\frac{n}{3}\) \frac{q^n}{1-q^n} \sin^2 nz =\frac{\T_1^3(z|\tau)}{12\T_1(3z|3\tau)}.
\end{equation}		
\end{thm}
\begin{proof}It is easy to verify that the elliptic function $\T_1^3(z|\tau)/{\T_1(3z|3\tau)}$ has only simple poles and a complete set of inequivalent poles is given by $\mathcal{P}=\{\pi/3, -\pi/3\}$. By Theorem~\ref{decomthm} and a simple calculation we have the  decomposition formula
\begin{align}\label{Kiepert:eqn23}
\frac{2\T_1^3(z|\tau)}{\sqrt{3}\T_1(3z|3\tau)}&=-2\frac{\T_1'(\frac{\pi}{3}|\tau)}
{\T_1(\frac{\pi}{3}|\tau)}+\frac{\T_1'(z+\frac{\pi}{3}|\tau)}{\T_1(z+\frac{\pi}{3}|\tau)}-\frac{\T_1'(z-\frac{\pi}{3}|\tau)}{\T_1(z-\frac{\pi}{3}|\tau)}\\
&=\frac{8\sin^3 z}{\sqrt{3} \sin 3z}-8\sqrt{3} \sum_{n=1}^\infty \(\frac{n}{3}\) \frac{q^n}{1-q^n} \sin^2 nz. \nonumber
\end{align}	
Dividing both sides of the above equation, we complete the proof of Theorem~\ref{Carlitzthm}.
\end{proof}
Dividing both sides of \reff{Kiepert:eqn22} by $x^2$ and then letting $z\to 0$, we 
arrive at Carlitz's identity \cite[Eq.(3.1)]{Carlitz1953}
\begin{equation}\label{Kiepert:eqn24}
1-9\sum_{n=1}^\infty \(\frac{n}{3}\) \frac{n^2 q^n}{1-q^n}=\frac{\eta^9(\tau)}{\eta^3(3\tau)}.
\end{equation}
The Glaisher--Ramanujan Eisenstein series $a(\tau)$ is defined by
\begin{equation}\label{Kiepert:eqn25}
a(\tau)=1+6\sum_{n=1}^\infty \(\frac{n}{3}\) \frac{q^n}{1-q^n}.
\end{equation}
J. W. Glaisher \cite{Glaisher1889} studied some arithmetic properties of $a(\tau)$ in 1889, and it was also discussed by Ramanujan in one of letter to Hardy, written from the nursing home,  Fitzroy House \cite[p.93]{Ramanujan1988}.
It is easily seen that
\begin{equation}\label{Kiepert:eqn26} \frac{\T_1'(\frac{\pi}{3}|\tau)}{\T_1(\frac{\pi}{3}|\tau)}=\frac{1}{\sqrt{3}}a(\tau).
\end{equation}
Using the above  equation we can rewrite \reff{Kiepert:eqn23} as
\begin{align}\label{Kiepert:eqn27}
\frac{2\T_1^3(z|\tau)}{\sqrt{3}\T_1(3z|3\tau)}=-\frac{2}{\sqrt{3}}a(\tau)+\frac{\T_1'(z+\frac{\pi}{3}|\tau)}{\T_1(z+\frac{\pi}{3}|\tau)}-\frac{\T_1'(z-\frac{\pi}{3}|\tau)}{\T_1(z-\frac{\pi}{3}|\tau)}
\end{align}	
Replacing $z$ by $z+\pi/3$ in the above equation and making a simple calculation, we deduce that
\begin{align}\label{Kiepert:eqn28}
\frac{2\T_1^3(z+\frac{\pi}{3}|\tau)}{\sqrt{3}\T_1(3z|3\tau)}=\frac{2}{\sqrt{3}}a(\tau)-\frac{\T_1'(z-\frac{\pi}{3}|\tau)}{\T_1(z-\frac{\pi}{3}|\tau)}+\frac{\T_1'(z|\tau)}{\T_1(z|\tau)}.
\end{align}
Replacing $z$ by $-z$ in the above equation and simplifying, we conclude that
\begin{align}\label{Kiepert:eqn29}
\frac{2\T_1^3(z-\frac{\pi}{3}|\tau)}{\sqrt{3}\T_1(3z|3\tau)}=\frac{2}{\sqrt{3}}a(\tau)+\frac{\T_1'(z+\frac{\pi}{3}|\tau)}{\T_1(z+\frac{\pi}{3}|\tau)}-\frac{\T_1'(z|\tau)}{\T_1(z|\tau)}.
\end{align}
Combining the above three equations, we find that
\[
\frac{2\T_1^3(z+\frac{\pi}{3}|\tau)}{\sqrt{3}\T_1(3z|3\tau)}
+\frac{2\T_1^3(z-\frac{\pi}{3}|\tau)}{\sqrt{3}\T_1(3z|3\tau)}
-\frac{2\T_1^3(z|\tau)}{\sqrt{3}\T_1(3z|3\tau)}
={2}{\sqrt{3}} a(\tau),
\]
or
\begin{equation}\label{Kiepert:eqn30}
\T_1^3\(z+\frac{\pi}{3}|\tau\)+\T_1^3\(z-\frac{\pi}{3}|\tau\)-\T_1^3(z|\tau)
=3a(\tau)\T_1(3z|3\tau).
\end{equation}
This beautiful identity  can be found in \cite[Eq.(3.15)]{Liu2007ADV}, but the derivation here is new.
\subsection{Some Lambert series identity due to Shen}
Shen \cite{Shen1994TAMS} proved  many amazing Lambert series identities by using the additive formulae of theta functions and the Fourier series expansions of the $12$ elliptic functions due to Jacobi, one of which is the following identity \cite[Eq.(1.6)]{Shen1994TAMS}.  
\begin{thm}\label{shenlambert:thm1} If $\T_1, \T_2, \T_3$ and $\T_4$ are the Jacobi theta functions, then, we have 
\begin{align}\label{Shen:eqn1}
&\cot x +\cot y -4\sum_{n=1}^\infty \frac{q^n}{1+q^n} \(\sin 2nx +\sin 2ny \)\\
&=\frac{2\T_3(0|\tau)\T_4(0|\tau)\T_1(x+y|2\tau)\T_4(x-y|2\tau)}{\T_1(x|\tau)\T_1(y|\tau)}.\nonumber
\end{align}		
\end{thm}
Next we  will use Theorem~\ref{decomthm} to give a new proof of Theorem~\ref{shenlambert:thm1}. Our proof is slightly different from that of \cite{Shen1994TAMS}.
\begin{proof} By a simple calculation, we find that the right-hand side of the above equation is an elliptic function with periods $\pi$ and $2\pi \tau.$ This elliptic function has only simple poles,  and the inequivalent poles of this function are $0$ and $\pi\tau$. Thus by Theorem~\ref{decomthm}, there exists three constants A, B and C independent of $x$ such that
\[
\frac{2\T_3(0|\tau)\T_4(0|\tau)\T_1(x+y|2\tau)\T_4(x-y|2\tau)}{\T_1(x|\tau)\T_1(y|\tau)}
=A+B \frac{\T_1'(x|2\tau)}{\T_1(x|2\tau)}+C\frac{\T_4'(x|2\tau)}{\T_4(x|2\tau)}
\]	
Multiplying both sides by $x$ and then letting $x\to 0$ and noting that
$2\T_1(y|2\tau)\T_1(y|2\tau)=\T_2(0|\tau)\T_1(y|\tau)$, we find that
$B=1$. In the same way, by multiplying both sides of the above equation by $(x-\pi\tau)$ and then letting $x \to \pi\tau$, we find that $C=-1.$ Hence we have 
\[
\frac{2\T_3(0|\tau)\T_4(0|\tau)\T_1(x+y|2\tau)\T_4(x-y|2\tau)}{\T_1(x|\tau)\T_1(y|\tau)}
=A+\frac{\T_1'(x|2\tau)}{\T_1(x|2\tau)}-\frac{\T_4'(x|2\tau)}{\T_4(x|2\tau)}.
\]
Setting $x=-y$ in the above equation and noting that $\T_1(0|\tau)=0$,  we immediately deduce that
\[
A=\frac{\T_1'(y|2\tau)}{\T_1(y|2\tau)}-\frac{\T_4'(y|2\tau)}{\T_4(y|2\tau)}.
\]
Thus we have 
\begin{align*}
&\frac{2\T_3(0|\tau)\T_4(0|\tau)\T_1(x+y|2\tau)\T_4(x-y|2\tau)}{\T_1(x|\tau)\T_1(y|\tau)}\\
&=\frac{\T_1'(x|2\tau)}{\T_1(x|2\tau)}-\frac{\T_4'(x|2\tau)}{\T_4(x|2\tau)}
+\frac{\T_1'(y|2\tau)}{\T_1(y|2\tau)}-\frac{\T_4'(y|2\tau)}{\T_4(y|2\tau)}.
\end{align*}
Using the trigonometric series expansions for the logarithmic derivatives of 
$\T_1(z|\tau)$ and $\T_4(z|\tau)$ in \reff{Kiepert:eqn5} and \reff{Kiepert:eqn6}
we have 
\[
\frac{\T_1'(x|2\tau)}{\T_1(x|2\tau)}-\frac{\T_4'(x|2\tau)}{\T_4(x|2\tau)}
=\cot z-4\sum_{n=1}^\infty \frac{q^n}{1+q^n} \sin 2nz,
\]
Combining the above two equations, we complete the proof of Theorem~\ref{shenlambert:thm1}.
\end{proof}	
In the same way we can also derive  the following three Lambert series identities, 
the first three of which belong to Shen, and the fourth seems to be new. 
\begin{thm}\label{shenlambert:thm2} We have
\begin{align}\label{Shen:eqn2}
&1+2\sum_{n=1}^\infty \frac{q^{n/2}}{1+q^n} (\cos 2nx+\cos 2ny)\\
&=\T_3(0|\tau)\T_4(0|\tau)\frac{\T_4(x+y|\tau)\T_4(x-y|
			2\tau)}{\T_4(x|\tau)\T_4(y|\tau)},\nonumber
\end{align}

\begin{align}\label{Shen:eqn3}
&\sum_{n=1}^\infty \frac{q^{n/2}}{1+q^n} (\cos 2nx-\cos 2ny)\\
&=-\T_3(0|\tau)\T_4(0|\tau)\frac{\T_1(x+y|2\tau)\T_1(x-y|
	2\tau)}{2\T_4(x|\tau)\T_4(y|\tau)},\nonumber
\end{align}

\begin{align}\label{Shen:eqn4}
&\csc x+\csc y+4\sum_{n=0}^\infty \frac{q^{n+1/2}}{1-q^{n+1/2}}
		(\sin (2n+1) x+\sin (2n+1)y)\\
		&=\T_2(0|\tau)\T_3(0|\tau)\frac{\T_1(\frac{x+y}{2}|\frac{\tau}{2})\T_2(\frac{x-y}{2}|\frac{\tau}{2})}
		{\T_1(x|\tau)\T_1(y|\tau)},\nonumber
\end{align}

\begin{align}
&4\sum_{n=0}^\infty \frac{q^{(2n+1)/4}}{1-q^{n+1/2}}( \sin
(2n+1)x+\sin(2n+1)y)\label{jell12}\\
&=\T_2(0|\tau)\T_3(0|\tau)\frac{\T_1(\frac{x+y}{2}|\frac{\tau}{2})\T_2(\frac{x-y}{2}|\frac{\tau}{2})}
{\T_4(x|\tau)\T_4(y|\tau)}.\nonumber
\end{align}
\end{thm}
Identities \reff{Shen:eqn2} and  \reff{Shen:eqn3} can be found in \cite[Eq.(1.11)]{Shen1994TAMS} and \reff{Shen:eqn4} is  a variant form of \cite[Eq.(2.4)]{Shen1994TAMS}.
\section{Acknowledgments}
I  would like to express deep appreciation to Li-Chien Shen for his  invaluable suggestions. I am  grateful to the referee for many helpful criticisms and suggestions to improve an earlier version of this paper. I also thanks for Dandan Chen for pointing out several misprints of an earlier version of this paper. 


\begin{thebibliography}{9}
	
\bibitem{AAR1999}	G. E. Andrews, R. Askey and R. Roy, Special functions, Encyclopedia of Mathematics And Its Applications 71, Cambridge University Press, Cambridge, 1999.	
	
\bibitem{AndrewsAskey1978}	G. E. Andrews, R. Askey, A simple proof of Ramanujan’s $_1\psi_1$, Aequationes Math. 18 (1978) 333--337.

\bibitem{Basoco1931} Basoco, M. A. 
On the trigonometric expansion of elliptic functions. developments,  Bull. Amer. Math. Soc. 31 (1937) 117--124.

\bibitem{Berndt} B. C. Berndt, Number Theory in the Spirit of Ramanujan, Amer. Math. Soc., Providence, RI, 2006.



\bibitem{Carlitz1953} L. Carlitz, Note on some partition formulae, Quart. J. Math. Oxford (2), 4(1953) 168--172.


\bibitem{Glaisher1889} J. W. L. Glaisher, On the function which denotes the excess of the number of divisors of a number which $\equiv 1 \pmod 3,$ over the number which $\equiv 2 \pmod 3$, Proc. London Math.Soc. 21 (1889) 395--402.

\bibitem{Ismail1977} M.E.H. Ismail, A simple proof of Ramanujan’s $_1\psi_1$ sum, 
Proc. Amer. Math. Soc. 63 (1977) 185--186.

\bibitem{Kiepert1879} L. Kiepert, 
Zur transformationstheorie der elliptischen functionen,
J. Reine Angew. Math. 87 (1879) 199--216.

\bibitem{Koornwinder} T. H. Koornwinder, On the equivalence of two fundamental theta identities. Anal. Appl. (Singap.) 12 (2014) 711--725.

\bibitem{Kronecker1881}
L. Kronecker, Zur Theorie der elliptischen Functionen, Monatsber. K. Akad. Wiss. zu Berlin (1881) 1165--1172.

\bibitem{KroneckerC1929}
L. Kronecker, Leopold Kronecker’s Werke, Bd. IV, B.G. Teubner, Leipzig, 1929, reprinted by Chelsea, New York, 1968.


\bibitem{LewisLiu2001} R. P. Lewis and Z.-G. Liu, 
The Borweins' cubic theta functions and $q$-elliptic functions, in: Symbolic Computation, Number Theory, Special Functions, Physics and Combinatorics, F. G. Garvan and M. E. H. Ismail (Eds.), Kluwer Acad. Publ., Dordrecht, 2001, 133--145. 

\bibitem{Liu20001residue} Z.-G. Liu,
  Residue theorem and theta function identities. Ramanujan J. 5 (2001) 129--151.


\bibitem{Liu2005ADV} Z.-G. Liu,	
A three-term theta function identity and its applications. Adv. Math. 195 (2005)  1--23.

\bibitem{Liu2007ADV} Z.-G. Liu,
An addition formula for the Jacobian theta function and its applications. Adv. Math. 212 (2007) 389--406.

\bibitem{Liu2007JRMS}Z.-G. Liu,
A theta function identity and the Eisenstein series on $\Gamma_0(5)$. J. Ramanujan Math. Soc. 22 (2007) 283--298.

\bibitem{Liu2009} Z.-G. Liu,
Addition formulas for Jacobi theta functions, Dedekind's eta function, and Ramanujan's congruences,
Pacific J. Math. 240 (2009) 135--150.

\bibitem{Liu2012JNT}Z.-G. Liu,
A theta function identity of degree eight and Eisenstein series identities. J. Number Theory 132 (2012) 2955--2966.

\bibitem{Liu2012IJNT}Z.-G. Liu,
Some inverse relations and theta function identities. Int. J. Number Theory 8 (2012) 1977--2002.

\bibitem{Ramanujan1988} S. Ramanujan,  The Lost Notebook and  Other
Unpublished papers, Narosa, New Delhi, 1988.
	
\bibitem{Shen1994}
L. C. Shen, On the modular equations of degree $3$, Proc. Amer. Math. Soc. 122 (1994) 1101--1114.

\bibitem{Shen1994TAMS} L.-C. Shen, On the additive formulae of the theta functions and a collection of Lambert series pertaining to the modular equations of degree 5, Trans. Amer. Math. Soc. 345 (1994) 323--345. 
	
\bibitem{W-W} E. T. Whittaker and G. N. Watson,
A course of modern analysis, 4th ed, Cambridge Univ. Press, Cambridge, 1966.
	
\bibitem{Winquist1969} L. Winquist, An elementary proof of $p(11m+6)\equiv 0 \pmod {11},$ J. Combin. Theory 6 (1969) 56--59.



\end{thebibliography}
\end{document}